\documentclass{amsart}

\newtheorem{theorem}{Theorem}[section]
\newtheorem{lemma}[theorem]{Lemma}

\theoremstyle{definition}
\newtheorem{definition}[theorem]{Definition}

\newtheorem{proposition}[theorem]{Proposition}
\newtheorem{corollary}[theorem]{Corollary}
\theoremstyle{remark}
\newtheorem{remark}[theorem]{Remark}

\numberwithin{equation}{section}

\begin{document}

\title{Compactness and rigidity of K\"{a}hler surfaces with constant scalar curvature}

\author{Hongliang Shao}
\address{Department of Mathematics, Capital Normal University, Beijing,
China 100048\newline
Department of Mathematics, University of California, San Diego, USA,}

\email{hongliangshao@foxmail.com}




\begin{abstract}
A compactness theorem is proved for a family of K\"{a}hler surfaces with constant
scalar curvature and volume bounded from below, diameter bounded
from above, Ricci curvature bounded and the signature bounded from below. Furthermore,
a splitting theorem and some rigidity theorems are proved for Einstein-Maxwell systems.
\end{abstract}
\maketitle
\section{introduction}
A metric $g$ on a manifold $M$ is called Einstein if it's Ricci curvature is constant, i.e.
$$\mathrm{Ric}=\sigma g.$$
The Einstein-Maxwell system was introduced as a generalization of Einstein manifolds which is
the Maxwell equation coupled with the mass free Einstein's gravitational field equation.
\begin{definition}(cf.\cite{Le10})
Let $(M,g)$ be a Riemannian manifold and $F$ be a $2$-form on $M$. If $(g,F)$ satisfies
\begin{equation}\label{eq:1.1}
\left\{
\begin{array}{c}
\mathrm{d}F=0 \\
\mathrm{d}^{\ast }F=0 \\
\overset{\circ }{\mathrm{Ric}}+\overset{\circ }{\left[ F\circ F\right] }=0%
\end{array}%
\right.
\end{equation}%
where $\overset{\circ }{\mathrm{Ric}}$ and $\overset{\circ }{\left[ F\circ F\right] }$
denote the trace free part of the Ricci tensor and $F\circ F$ with respect
to $g$ respectively,
then we say $(g,F)$ satisfies the Einstein-Maxwell equations and $(M,g,F)$ is
an Einstein-Maxwell system.
\end{definition}
The first and second equations in (\ref{eq:1.1}) are the electromagnetic
field equations (Maxwell equations) and $F$ is the electromagnetic field intensity.
Einstein Maxwell equations had been extensively studied in the
literatures of both physics and mathematics (cf. \cite{Le10} and
references in it).

The convergence of Einstein manifolds in the Gromov-Hausdorff
sense has been studied by various authors (cf. \cite{An90},\cite{An89},\cite{BKN89},\cite{CT06},\cite{Ti90} etc).
In \cite{Sh13}, the compactness of a family of Einstein Yang-Mills systems, which are
special solutions to the Einstein-Maxwell equations, was studied.
Inspired by an observation of C.LeBrun that all K\"{a}hler surfaces with constant scalar
curvature can be considered as solutions to the Einstein-Maxwell equations, we are interested
in the compactness of K\"{a}hler surfaces with constant scalar curvature. The main theorem is
the following.

\begin{theorem}\label{thm:1}
Let $(M_{i},J_{i},g_{i})$ be a sequence of K\"{a}hler surfaces with constant scalar curvature.
Assume that there are constants $C>0,v>0, D>0, \Lambda>0$ independent
of $i$ such that
\begin{itemize}
\item[(\romannumeral1)] $\mathrm{Vol}_{M_{i}}\geq v>0,$ and $\mathrm{diam}_{M_{i}}\leq D,$
\item[(\romannumeral2)]$|\mathrm{Ric}_{g_{i}}|\leq \Lambda,$
\item[(\romannumeral3)] $\tau\left( M_{i}\right) \geq -C$.
\end{itemize}
Then a subsequence of $(M_{i},J_{i},g_{i})$ converges, without changing
the subscripts, in the Gromov-Hausdorff sense, to a connected
orbifold $(M_{\infty },J_{\infty},g_{\infty })$ with finite singular points
$\{p_{k}\}_{k=1}^{N},$ each having a neighborhood homeomorphic to the cone
$C\left( S^{n-1}/\Gamma _{k}\right) ,$ with $\Gamma _{k}$ a finite
subgroup of $O\left( n\right)$ . The metric $g_{\infty }$ is a
$C^{0}$  orbifold metric on $M_{\infty },$ which is smooth and K\"{a}hler
off the singular points and has constant scalar curvature.
\end{theorem}

\begin{remark}
If we replace the condition $\tau\left( M_{i}\right) \geq -C$ by $|W^{-}|\leq C$,
then the limit space is a smooth K\"{a}hler manifold with constant scalar curvature.
\end{remark}

\begin{remark}
This convergence result holds for Einstein-Maxwell systems with constant scalar curvature
if we replace the condition $\tau\left( M_{i}\right) \geq -C$ by a $L^{\frac{n}{2}}$-bound
for the Riemannian curvature tensor. Moreover, if the underlying manifolds are of odd-dimension,
then the limit space is a smooth manifold.
\end{remark}
This paper is organized as follows. Section 2 is devoted to present some basic properties of Einstein-Maxwell
systems especially 4-dimensional manifolds. In section 3 we shall complete the proof of Theorem \ref{thm:1}.
In section 4, a splitting theorem is proved for odd-dimensional Einstein-Maxwell systems with $\mathrm{Ric}-\eta=0$.
In section 5, some rigidity properties of Einstein-Maxwell systems are studied. In particular,
we will show that a generic K\"{a}hler surface with constant scalar curvature and positive
isotropic curvature must be biholomorphic to $\mathbb{C}P^{2}$ with constant holomorphic sectional curvature.
\bigskip
\section{Basic Properties of Einstein-Maxwell systems}
Assume $(M,g,F)$ is an Einstein-Maxwell system. We rewrite the Einstein-Maxwell equations as follows:
\begin{equation}\label{eq:2.1}
\left\{
\begin{array}{c}
\mathrm{Ric}-\eta=fg\\
\Delta_{\mathrm{d}}F=0.
\end{array}
\right.
\end{equation}
Where $\eta=-F\circ F$, and $f=R-|F|^{2}$ is a smooth function on $M$, and $\Delta_{\mathrm{d}}=\mathrm{d}\mathrm{d}^{*}+\mathrm{d}^{*}\mathrm{d}$
is the Hodge Laplace. $\Delta_{\mathrm{d}}F=0$ is equivalent to
\begin{equation*}
\left\{
\begin{array}{c}
\mathrm{d}F=0 \\
\mathrm{d}^{\ast }F=0
\end{array}%
\right.
\end{equation*}

The Schur lemma states that if a Riemannian metric $g$ satisfies
$\mathrm{Ric}(g)=\frac{1}{n}Rg$ for $n\geq 3$, then the scalar curvature $R$ is constant.
Unfortunately, the generalized system (\ref{eq:2.1}) does not own this nice property.
However, 4-dimensional Einstein Maxwell systems possess a privilege that
the scalar curvature of $g$ turns out bo be constant.
\begin{lemma}\label{lem:2.1}
Let $M$ be a complete Riemannian manifold with a Riemannian metric $g$ and $F$ be a $2$-form on $M$.
Assume that $(g,F)$ satisfies the Einstein-Maxwell equations (\ref{eq:2.1}),
then there is a constant $C$ such that
$$
(4-2n)R+(n-4)|F|^{2}=C.
$$
When $n=4$, the scalar curvature $R$ of the metric $g$ is constant.

When $n=2$, $|F|$ is constant. Moreover,if $M$ is compact, then $F=\pm |F|\mathrm{d}\mu$ and $\overset{\circ}{\mathrm{Ric}}=0$.

When $n\neq 2\;\text{ or }4 $, $g$ has constant scalar curvature if and only if $|F|$ is constant.
\end{lemma}
\begin{proof}
Taking covariant derivative to the first equation of (\ref{eq:2.1}) we get
$$
\nabla_{i}R_{jk}-\nabla_{i}\eta_{jk}=\nabla_{i}fg_{jk},
$$
and
$$
\nabla_{j}R_{ik}-\nabla_{j}\eta_{ik}=\nabla_{j}fg_{ik}.
$$
From these we obtain
\begin{equation}\label{eq:2.4}
\nabla_{i}R_{jk}-\nabla_{j}R_{ik}-(\nabla_{i}\eta_{jk}-\nabla_{j}\eta_{ik})
=\nabla_{i}fg_{jk}-\nabla_{j}fg_{ik},
\end{equation}
Since
\begin{eqnarray*}
g^{jk}\nabla _{j}\eta _{ik} &=&g^{jk}\nabla _{j}\left(
g^{pq}F_{ip}F_{kq}\right) \\
&=&g^{jk}g^{pq}F_{ip}\nabla _{j}F_{kq}+g^{jk}g^{pq}F_{kq}\nabla _{j}F_{ip} \\
&=&-g^{pq}F_{ip}d^{\ast }F_{p}+g^{jk}g^{pq}F_{ip}\nabla _{j}F_{kq} \\
&=&-g^{jk}g^{pq}F_{kq}( \nabla _{i}F_{pj}+\nabla _{p}F_{ji}) \\
&=&\frac{1}{4}\nabla _{i}\left\vert F\right\vert ^{2},
\end{eqnarray*}%
where we have used the second equation of the Einstein-Maxwell equations
(\ref{eq:2.1}) and the second Bianchi identity.
Taking trace by $g^{jk}$ to both sides of equation (\ref{eq:2.4}), thus we have
$$
\frac{1}{2}\nabla_{i}R-\frac{3}{4}\nabla_{i}|F|^{2}=(n-1)\nabla_{i}f.
$$
On the other hand, taking trace of the first equation of (\ref{eq:2.1}) we have
$$
R-|F|^{2}=nf.
$$
and then take derivative, we have
$$
\nabla_{i}R-\nabla_{i}|F|^{2}=n\nabla_{i}f.
$$
Thus it is easy to see that
$$
\nabla((4-2n)R+(n-4)|F|^{2})=0.
$$
Consequently, there exists a constant $C$ such that
$$
(4-2n)R+(n-4)|F|^{2}=C.
$$
In particular, if $n=4$, we derive that the scalar curvature of the metric is constant.
When $n=2$, $|F|$ is constant and $F$ is harmonic. If furthermore $M$
is compact, by Hodge theory, there is a unique harmonic form in each cohomology group up to scaling.
There is an isomorphism
$$H^{2}(M;\mathbb{R})\cong\mathcal{H}^{2},$$
where $\mathcal{H}^{2}$ is the space of harmonic 2-forms.

On the other hand, since $M$ is 2-dimensional, by the Poincar\'{e} dual theorem,
 $$H^{2}(M;\mathbb{R})\cong H_{0}(M)\bigotimes\mathbb{R}\cong\mathbb{R}.$$
Thus $F=\pm|F|\mathrm{d}\mu$. In this case, $\overset{\circ}{\mathrm{Ric}}=0.$

When $n\neq 2,4$, the scalar curvature is constant if and only if the norm of $F$ is constant.
\end{proof}

We note here that this theorem can be interpreted from another point of view (cf.\cite{Le10}).
Let $M$ be a smooth manifold. Denote
$$
\mathcal{M}_{V}=\{\;\text{Riemannian metrics with volume}\; V \;\text{on}\; M\}.
$$
Assume $\theta \in H^{2}(M; \mathbb{R})$ is a fixed De Rham class of $M$. Define a functional
$$
\begin{array}{c}
\mathcal{M}_{V}\times\theta\longrightarrow\mathbb{R}\\
(g,F)\mapsto\int_{M}(R+|F|^{2})\mathrm{d}\mu.
\end{array}
$$
In the category of compact Riemannian manifolds,
Einstein-Maxwell equations can be interpreted as the Euler-Lagrange equations of
this functional. In particular, when $n=4$, $\int|F|^{2}d\mu$ is conformal invariant, which implies that
critical points of the above functional are just the critical points of the Yamabe functional,
thus must have constant scalar curvature.

Suppose now that $M$ is an oriented Riemannian $4$-manifold and $g$
is a Riemannian metric on $M$. The Hodge star
operator $*:\Omega^{2}M\rightarrow \Omega^{2}M$ is defined by
$$
\alpha\wedge\ast\beta=(\alpha, \beta)_{g}d\mu_{g}.
$$
Where $\alpha,\beta\in\Omega^{2}M$, $(\cdot,\cdot)_{g}$ denotes the induced inner product
on $\Omega^{2}M$ and $d\mu_{g}$ denotes the volume form of $g$. It is well known that
$\ast^{2}=1_{\Omega^{2}}$. Then we have the decomposition of 2-forms
into self-dual and anti-self-dual forms, defined to be the $\pm1$ eigenspaces
of the Hodge star operator. We denote them by
$\Omega_{M}^{+},\Omega_{M}^{-}$ respectively.
Accordingly, the curvature operator $\mathrm{Rm}(g)$ decomposes as follows (cf.\cite{Be87}):
$$
\mathrm{Rm}=\begin{pmatrix}
W^{+}+\frac{R}{12}\mathrm{I} & \overset{\circ}{\mathrm{Ric}}\\
\overset{\circ}{\mathrm{Ric}} & W^{-}+\frac{R}{12}\mathrm{I}
\end{pmatrix}.
$$
Where the trace free Ricci curvature $\overset{\circ}{\mathrm{Ric}}=\mathrm{Ric}-\frac{R}{4}g$ acts on 2-forms by
$$\overset{\circ}{\mathrm{Ric}}(\alpha)=\overset{\circ}{R}_{ik}\alpha^{k}_{j}-\overset{\circ}{R}_{jk}\alpha^{k}_{i}.$$
The Bianchi identity implies that
$$\mathrm{tr}W^{+}=\mathrm{tr}W^{-}.$$
If $M$ is compact, by Hodge theory that there is a unique harmonic form in each cohomology group up to
scaling. Then there is an isomorphism
$$H^{2}(M;\mathbb{R})\cong\mathcal{H}^{+}\bigoplus\mathcal{H}^{-},$$
where
$$\mathcal{H}^{+}=\{ \text{self-dual harmonic 2-forms}\},$$
and
$$\mathcal{H}^{-}=\{ \text{anti-self-dual harmonic 2-forms}\}.$$
The signature $\tau$ of $M$ is defined by
$$\tau=b_{+}-b_{-},$$
here $b_{\pm}=\mathrm{dim}\mathcal{H}^{\pm}.$

The Hirzebruch signature theorem tells us that
$$\tau=\frac{1}{12\pi^{2}}\int_{M}|W^{+}|^{2}-|W^{-}|^{2}\mathrm{d}\mu.$$

We present a simple lemma for 4-manifolds first.
\begin{lemma}\label{lem:2.2}
Let $F$ be a harmonic $2$-form on a 4-dimensional oriented manifold $M,$
then
$$\overset{\circ}{\eta}=-2F^{+}\circ F^{-}$$
and $$\left\vert \overset{\circ }{\eta }\right\vert ^{2}=\left\vert
F^{+}\right\vert ^{2}\left\vert F^{-}\right\vert ^{2},$$
where $F^{+}$ and $F^{-}$ denote the self-dual part and the anti-self-dual part
of $F$ respectively.
\end{lemma}

\begin{proof}
Fix a point $x\in M$ and choose local coordinates such that $g_{ij}=\delta
_{ij}$ and $F$ has been skew-diagonalized at $x,$ with out loss of
generality, we may assume $F=\mu \mathrm{d}x^{1}\Lambda \mathrm{d}x^{2}+\nu \mathrm{d}x^{3}\Lambda
\mathrm{d}x^{4},$ where $\mu ,\nu \in \mathbb{R}.$ It is easily to compute that
\begin{equation*}
\eta = \mu ^{2} \left(
\mathrm{d}x^{1}\otimes \mathrm{d}x^{1}+\mathrm{d}x^{2}\otimes \mathrm{d}x^{2}\right) + \nu^{2}\left( \mathrm{d}x^{3}\otimes \mathrm{d}x^{3}+\mathrm{d}x^{4}\otimes \mathrm{d}x^{4}\right)
\end{equation*}
and
\begin{equation*}
\overset{\circ }{\eta }=\frac{1}{2}\left( \mu ^{2}-\nu ^{2}\right) \left(
\mathrm{d}x^{1}\otimes \mathrm{d}x^{1}+\mathrm{d}x^{2}\otimes \mathrm{d}x^{2}\right) +\frac{1}{2}\left( \nu
^{2}-\mu ^{2}\right) \left( \mathrm{d}x^{3}\otimes \mathrm{d}x^{3}+\mathrm{d}x^{4}\otimes \mathrm{d}x^{4}\right)
\end{equation*}%
On the other hand,%
\begin{equation*}
F^{+}=\frac{1}{2}\left( \mu +\nu \right) \left( \mathrm{d}x^{1}\Lambda
\mathrm{d}x^{2}+\mathrm{d}x^{3}\Lambda \mathrm{d}x^{4}\right),
\end{equation*}%
and%
\begin{equation*}
F^{-}=\frac{1}{2}\left( \mu -\nu \right) \left( \mathrm{d}x^{1}\Lambda
\mathrm{d}x^{2}-\mathrm{d}x^{3}\Lambda \mathrm{d}x^{4}\right).
\end{equation*}%
Then we have
$$\overset{\circ}{\eta}=-2F^{+}\circ F^{-},$$
and
\begin{equation*}
\left\vert \overset{\circ }{\eta }\right\vert ^{2}=\left( \mu ^{2}-\nu
^{2}\right) ^{2},
\end{equation*}%
\begin{equation*}
\left\vert F^{+}\right\vert ^{2}=\left( \mu +\nu \right) ^{2},
\end{equation*}%
\begin{equation*}
\left\vert F^{-}\right\vert ^{2}=\left( \mu -\nu \right) ^{2}.
\end{equation*}%
Finally we have the identity%
\begin{equation*}
\left\vert \overset{\circ }{\eta }\right\vert ^{2}=\left\vert
F^{+}\right\vert ^{2}\left\vert F^{-}\right\vert ^{2}.
\end{equation*}
\end{proof}

For any K\"{a}hler surface, $|W^{+}|^{2}=\frac{R^{2}}{24}$.
In particular, if the K\"{a}hler metric is of constant scalar curvature $R$, then the self-dual
Weyl tensor $W^{+}$ can be written as
$$W^{+}=\begin{pmatrix}
\frac{R}{6}&  &  \\
  &\-\frac{R}{12}&  \\
  &  &-\frac{R}{12}
\end{pmatrix}.$$
Thus a lower bound of $\tau$ gives an upper bound of $\|W^{-}\|_{2}$ and then
an upper bound of $\|W\|_{2}$ for K\"{a}hler surfaces with constant scalar curvature.
In particular, an upper bound of $\|\overset{\circ}{W^{-}}\|$ gives an upper bound for
the Riemannian curvature tensor $\mathrm{Rm}$ by virtue of $$\mathrm{tr}W^{+}=\mathrm{tr}W^{-}.$$

Based on the above discussion, we have the following easy corollary.
\begin{corollary}
Let $M$ be an oriented 4-dimensional manifold, and $(M,g,F)$ be an Einstein-Maxwell system.
\begin{itemize}
\item [(1)]Suppose
$b_{+}=0$ or $b_{-}=0$, then $(M,g,F)$ reduces to an Einstein manifold, and thus satisfies
$$2\chi(M)\pm3\tau(M)\geq0.$$
\item [(2)] If $(M,g)$ is a K\"{a}hler surface with $f=R-|F|^{2}\geq0$, then
$$2\chi(M)+\tau(M)\geq0.$$
\end{itemize}
\end{corollary}
\begin{proof}
(1).This is clear since if $b_{+}=0$ or $b_{-}=0$, $F^{+}=0$ or $F^{-}=0$, the above lemma tells us that
$\overset{\circ}{\mathrm{Ric}}=\overset{\circ}{\eta}=0$.

(2).Since $(M,g)$ is a K\"{a}hler surface, thus $|W^{+}|^{2}=\frac{R^{2}}{24}$. From the Gauss-Bonnet-Chern formula
and Hirzebruch signature formula, we have
\begin{eqnarray*}
& &2\chi \left( M\right) +3\tau \left( M\right) \\
&=&\frac{1}{4\pi ^{2}}\int_{M}\left( \frac{R^{2}}{24}+2\left\vert
W^{+}\right\vert ^{2}-\frac{1}{2}\left\vert \overset{\circ }{\mathrm{Ric}}\right\vert
^{2}\right) \mathrm{d}\mu \\
&=&\frac{1}{8\pi ^{2}}\int_{M}\left( \frac{R^{2}}{4}-\left\vert \overset{%
\circ }{\mathrm{Ric}}\right\vert ^{2}\right)\mathrm{d}\mu \\
&=&\frac{1}{8\pi ^{2}}\int_{M}\left( \frac{1}{4}\left( \frac{1}{2}\left\vert
F\right\vert ^{2}+4f \right) ^{2}-\frac{1}{4}\left\vert \overset{\circ }%
{\eta }\right\vert ^{2}\right)\mathrm{d}\mu \\
&=&\frac{1}{32\pi ^{2}}\int_{M}\left( \left( \frac{1}{2}\left\vert
F\right\vert ^{2}+4f \right) ^{2}-\left\vert F^{+}\right\vert
^{2}\left\vert F^{-}\right\vert ^{2}\right) \mathrm{d}\mu\\
&\geq&\frac{1}{32\pi ^{2}}\int_{M}\left( 2f|F|^{2}+16f^{2}\right) \mathrm{d}\mu\\
&\geq&0.
\end{eqnarray*}%
\end{proof}

As we have shown that the Einstein-Maxwell equations on a 4-manifold imply the scalar
curvature is constant. On the converse, a remarkable observation by C.LeBrun (cf.\cite{Le10})
asserts that any K\"{a}hler metric with constant scalar curvature
on a K\"{a}hler surface can be interpreted as a solution of the Einstein-Maxwell equations.

\begin{proposition}\label{prop:1}(LeBrun \cite{Le10})
Let $(M,g,J)$ be a K\"{a}hler surface with K\"{a}hler form
 $\omega=g(J\cdot,\cdot)$ and Ricci form $\rho=Ric(J\cdot,\cdot)$. Suppose the scalar curvature $R$ is constant.
Set $$\overset{\circ }{\rho}=\rho-\frac{R}{4}\omega$$
and $$F_{a}=a\omega+\frac{1}{2a}\overset{\circ }{\rho}$$
for any constant $a>0$.
Then $(g,F_{a})$ solves the Einstein-Maxwell equations.
\end{proposition}
\begin{proof}
In this special case, $F_{a}^{+}=a\omega$ and $F_{a}^{-}=\frac{1}{2a}\overset{\circ}{\rho}$.
As we have shown that
$$
\begin{array}{c}
\overset{\circ}{\eta}=-2F_{a}^{+}\circ F_{a}^{-}\\
=\overset {\circ}{\mathrm{Ric}}.
\end{array}
$$

On the other hand, since the metric has constant scalar curvature, the second Bianchi identity implies
$$\mathrm{d}^{*}\rho=0.$$
Henceforth
$$\mathrm{d}^{*}\overset{\circ}{\rho}=\mathrm{d}^{*}(\rho-\frac{R}{4}\omega)=0,$$
and
$$\mathrm{d}(\rho -\frac{R}{4}\omega)=0.$$
Then
$$\mathrm{d}F_{a}=\mathrm{d}(a\omega+\frac{1}{2a}\overset{\circ}{\rho})=0,$$
and
$$\mathrm{d}^{\ast}F_{a}=\mathrm{d}^{\ast}a\omega+\frac{1}{2a}\mathrm{d}^{*}\overset{\circ}{\rho}=0.$$
Thus we have proved that $(g,F_{a})$ is a solution to Einstein-Maxwell equations.
\end{proof}

Given tensors $\xi $ and $\zeta ,$ $\xi \ast \zeta $ denotes some
linear combination of contractions of $\xi \otimes \zeta $ in this section.
\begin{lemma}(cf.Lemma $2.3$ in \cite{Sh13})
\label{lem:2.3} Let $(M,g)$ be a Riemannian manifold. Suppose $F$ is a $2$-form
on $M$ such that $(g,F)$ is a solution to the Einstein-Maxwell equations (\ref{eq:2.1}).
Let $u:U\rightarrow \mathbb{R}^{n}$ be a harmonic coordinate of the
underlying manifold $M.$ Then in this coordinate, $g$ and $F$ satisfies
\begin{equation}\label{eq:2.2}
-\frac{1}{2}g^{kl}\frac{\partial ^{2}g_{ij}}{\partial u^{k}\partial u^{l}}%
-Q_{ij}\left( g,\partial g\right)
-\frac{1}{2}g^{kl}F_{ik}F_{jl}-f g_{ij}=0
\end{equation}%
\begin{equation}\label{eq:2.3}
g^{kl}\frac{\partial ^{2}F_{ij}}{\partial u^{k}\partial u^{l}}%
+P_{ij}(g,\partial g,\partial F)+T_{ij}(g,\partial g,F)=0
\end{equation}%
where
\begin{equation*}
Q\left( g,\partial g\right) =\left( g^{-1}\right) ^{\ast 2}\ast
\left(
\partial g\right) ^{\ast 2},
\end{equation*}%
\begin{equation*}
P(g,\partial g,\partial F)=\left( g^{-1}\right) ^{\ast 2}\ast
\partial g\ast
\partial F,
\end{equation*}%
and%
\begin{equation*}
T(g,\partial g,\partial ^{2}g,F)=\left( g^{-1}\right) ^{\ast 3}\ast
\left(
\partial g\right) ^{\ast 2}\ast F+\left( g^{-2}\right) ^{\ast 2}\ast
\partial ^{2}g\ast F.
\end{equation*}
\end{lemma}
\bigskip
\section{Convergence of K\"{a}hler surfaces with constant scalar curvature}
In this section we are going to prove a convergence theorem for
K\"{a}hler surfaces with constant scalar curvature.
As we have shown that every K\"{a}hler surface with constant scalar curvature
can be considered as a solution to the Einstein-Maxwell equations. Then we can
use elliptic estimates to the Maxwell equations to obtain the regularity of the metric.

Before we prove the convergence theorem, we shall review two propositions firstly.
\begin{proposition}
\label{prop:3.1}
(Proposition 2.5, Lemma 2.2 and Remarks 2.3
in \cite{An90})
 Let $(M,g)$ be a Riemannian manifold with
\begin{equation*}
\left\vert \mathrm{Ric}_{M}\right\vert \leq \Lambda ,\ \ \mathrm{diam}%
_{M}\leq D\ \ and\ \ \mathrm{Vol}_{B\left( r\right) }\geq v_{0}>0
\end{equation*}%
for a ball $B\left( r\right) $.  Then, for any $C>1$, there exist
positive
constants $ \sigma=\sigma \left( \Lambda ,v_{0},n,D \right)$, $\epsilon =\epsilon \left( \Lambda ,v_{0},n,C\right) $ and $%
\delta =\delta \left( \Lambda ,v_{0},n,C\right) ,$ such that for any $%
1<p<\infty ,$ one can obtain a $\left( \delta ,\sigma
,W^{2,p}\right) $ adapted atlas on the union $U$ of those balls
$B\left( r\right) $ satisfying
\begin{equation*}
\int_{B\left( 4r\right) }\left\vert \mathrm{Rm}\right\vert ^{\frac{n}{2}%
}d\mu \leq \epsilon .
\end{equation*}%
More precisely, on any $B\left( 10\delta \right) \subset U$, there
is a
harmonic coordinate chart such that for any $1<p<\infty ,$%
\begin{equation*}
C^{-1}\delta _{ij}\leq g_{ij}\leq C\delta _{ij},
\end{equation*}%
and%
\begin{equation*}
\left\Vert g_{ij}\right\Vert _{W^{2,p}}\leq C.
\end{equation*}
\end{proposition}

A local version of the Cheeger-Gromov compactness theorem
(cf. Theorem 2.2 in \cite{An89}, Lemma 2.1 in \cite{An90} and
\cite{GW88}) is the following.
\begin{proposition}
\label{prop:3.2}\bigskip Let $V_{i}$ be a sequence of domains in closed $%
C^{\infty }$ Riemannian manifolds $(M_{i},g_{i})$ such that $V_{i}$
admits an adapted harmonic atlas $\left( \delta ,\sigma ,C^{l,\alpha
}\right) $ for a constant $C>1$. Then there is a subsequence which
converges uniformly on compact subsets in
the $C^{l,\alpha^{\prime }}$ topology, $\alpha^{\prime }<\alpha$, to a $%
C^{l,\alpha }$ Riemannian manifold $V_{\infty }.$
\end{proposition}

Now we are ready to prove the compactness theorem \ref{thm:1}.
\begin{proof}[proof of Theorem \ref{thm:1}]
As we discussed in section 2, for any K\"{a}hler surface with constant scalar curvature,
a lower bound of the signature $\tau(M_{i})$ gives an upper bound for the $L^{2}$-norm of the Weyl
tensors. This together with the condition $|Ric_{g_{i}}|\leq \Lambda$ and $\mathrm{diam}_{M_{i}}\leq D$
gives us a bound of $L^{2}$-norm for the curvature tensor. In fact, using Bishop volume
comparison theorem (cf.\cite{CLN06} for example), we get an upper bound for the volume of $M_{i}$,
$$\mathrm{Vol}_{M_{i}}\leq \mathrm{Vol}_{\frac{-\Lambda}{3}}(B(D)).$$
Here $\mathrm{Vol}_{\frac{-\Lambda}{3}}(B(D))$ is the volume of ball of radius $D$ in the space form
of constant curvature $\frac{-\Lambda}{3}$.
Then
\begin{eqnarray*}
\int_{M_{i}}\left(|\mathrm{Rm}(g_{i})|^{2}\right)\mathrm{d}\mu_{i}&=&\int_{M_{i}}\left(\frac{R^{2}}{6}+2|\overset{\circ}{\mathrm{Ric}}(g_{i})|^{2}+4|W(g_{i})|^{2}\right)\mathrm{d}\mu_{i}\\
&=&\int_{M{i}}\left(\frac{R^{2}}{6}+2|\overset{\circ}{\mathrm{Ric}}(g_{i})|^{2}+8|W_{+}(g_{i})|^{2}\right)\mathrm{d}\mu_{i}-4\tau(M_{i})\\
&=&\int_{M{i}}\left(\frac{R^{2}}{6}+2|\overset{\circ}{\mathrm{Ric}}(g_{i})|^{2}+\frac{2R^{2}}{3}\right)\mathrm{d}\mu_{i}-4\tau(M_{i})\\
&\doteq&C_{1}(\Lambda, D,C).
\end{eqnarray*}
From Proposition \ref{prop:3.1} and Proposition \ref{prop:3.2}, Theorem 2.6 in \cite{An90},
there is a subsequence of $(M_{i},g_{i})$ converges to a Riemannian orbifold
$(M_{\infty},g_{\infty})$ with finite isolated singular points in the Gromov-Hausdorff
sense. Furthermore, $g_{i}$ converge to $g_{\infty}$ in $C^{1,\alpha}$ topology on the
regular part in the Cheeger-Gromov sense.
Similar as the proof of Theorem 1.1 in \cite{Sh13}, we will use the Einstein-Maxwell equations
under the harmonic coordinates to get the regularity of the metric.

Set $F_{i}=\omega_{i}+\frac{1}{2}\overset{\circ}{\rho_{i}}$, where $\omega_{i}$
is the K\"{a}hler form corresponding to $g_{i}$, then $(g_{i},F_{i})$
satisfies the Einstein-Maxwell equations,
$$\left\{
\begin{array}{c}
\mathrm{Ric}(g_{i})-\eta(g_{i},F_{i})=f_{i}g_{i}\\
\Delta_{d}F_{i}=0,
\end{array}
\right.$$
where
\begin{eqnarray*}
f_{i}&=&\frac{1}{n}(R_{i}-|F_{i}|^{2})\\
&=&R_{i}-(4-\frac{1}{4}|\overset{\circ}{Ric(g_{i})}|^{2}),
\end{eqnarray*}
which is uniformly bounded since the Ricci curvature of $g_{i}$
are uniformly bounded.

Similar as the proof of Theorem 1.1 in \cite{Sh13}, for a given $r>0,$ let $\{B_{x_{k}}^{i}(r)\}$ be a family of metric
balls of
radius $r$ such that $\{B_{x_{k}}^{i}(r)\}$ covers $(M_{i},g_{i})$, and   $B_{x_{k}}^{i}(\frac{r}{2%
})$ are disjoint. Denote%
\begin{equation*}
G_{i}(r)=\cup \left\{ \left. B_{x_{k}}^{i}(r)\right\vert
\int_{B_{x_{k}}^{i}(4r)}\left\vert \mathrm{Rm}\left( g_{i}\right)
\right\vert ^{2}d\mu _{i}\leq \epsilon \right\} ,
\end{equation*}%
where $\epsilon =\epsilon \left( \Omega ,v,D,n,C_{0}\right) >0$ is
obtained in Proposition 2.4 in \cite{Sh13} for a constant $C_{0}>1$. So
$G_{i}(r)$ are covered by a  $\left( \delta ,\sigma ,W^{2,p}\right)
$ $\left( \text{for any }1<p<\infty \right) $ adapted atlas with the
harmonic radius uniformly bounded from below. In these coordinates we have $W^{2,p}$
bounds for the metrics, i.e.
\begin{equation*}
C_{3}^{-1}\delta _{jk}\leq g_{i,jk}\leq C_{3}\delta _{jk},
\end{equation*}%
and%
\begin{equation*}
\left\Vert g_{i}\right\Vert _{W^{2,p}}\leq C_{3}.
\end{equation*}%
And then follows
$$|\mathrm{Rm}(g_{i})|_{L^{p}}\leq C$$
for any $1<p<\infty$.

And the Sobolev embedding theorem tell us that the $C^{1,\alpha}$-norm
of $g_{i}$ is uniformly bounded for all $0<\alpha \leq 1-\frac{n}{2p}$.

On the other hand, since $|\mathrm{Ric}_{g_{i}}|\leq C$
and $\mathrm{diam}_{M_{i}}\leq D$, from the volume comparison theorem
we know that the volume of $M_{i}$ is uniformly bounded.
Since here $F_{i}=\omega_{i}+\frac{1}{2}\overset{\circ}{\rho_{i}}$,
then it follows
$$
\|F_{i}\|_{L^{2p}}=(\int_{M_{i}}(4+\frac{1}{4}|\overset{\circ}{\mathrm{Ric}(g_{i})}|^{2})^{p}\mathrm{d}\mu_{g_{i}})^{\frac{1}{p}}\leq C,
$$
for some constant $C$.

By applying $L^{p}$-esimates for the elliptic differential equations (cf.\cite{GT77}) to (\ref{eq:2.3})
\begin{equation*}
\| F_{i}\| _{W^{2,2p}}\leq C,
\end{equation*}
Then
$$\|F_{i}\|_{C^{1,\alpha}}\leq C,$$
by the soblev embedding again.

From
$$f_{i}=\frac{1}{4}(R(g_{i}+|F_{i}|^{2})),$$
and the fact $R(g_{i})$ is constant, we know that
$$\|f_{i}\|_{C^{1,\alpha}}\leq C.$$
Now use the Schauder estimate for elliptic differential equations to (\ref{eq:2.1})
we can obtain
\begin{equation*}
\left\Vert g_{i}\right\Vert _{C^{2,\alpha }}\leq C(n,\left\Vert
g_{i}\right\Vert _{C^{1,\alpha }},\left\Vert F_{i}\right\Vert
_{C^{\alpha }}^{2})\leq C,
\end{equation*}
Then, standard elliptic theory
implies all the covariant derivatives of the curvature tensor have uniform bounds.
By Proposition \ref{prop:3.2} there is a subsequence of $G_{i}(r)$ converges in the
$C^{\infty}$ topology to an open manifold $G_{r}$ with a smooth metric $g_{r}$ which is
K\"{a}hler.

The rest of the proof is same with Theorem 1.1 in \cite{Sh13}.

For any compact subset
$K\subset \subset M_{\infty }^{0}$, there are embeddings $\Phi
_{K}^{i}:K\rightarrow M_{i}$ such that $\Phi
_{K}^{i,\ast}\circ J_{i}\circ \Phi_{K,\ast}^{i}
\rightarrow J_{\infty}%
$ for $i\gg 1$, and
\begin{equation*}
\Phi _{K}^{i,\ast }g_{i}\rightarrow g_{\infty },\ \ \ \Phi
_{K}^{i,\ast }F_{i}\rightarrow F_{\infty },
\end{equation*}%
when $i\rightarrow \infty $ in the $C^{\infty }$-sense.
\end{proof}

\begin{remark}
In particular, if all the anti-self-dual weyl tensors are uniformly bounded,
i.e. $|W^{-}|\leq C$ for some constant $C$, then the limit space is smooth and K\"{a}hler.
\end{remark}
\bigskip
\section{A splitting theorem for Einstein-Maxwell systems}
At first we mention that $\eta=-F\circ F$ is nonnegativly definite. If the function $f$ in the
Einstein-Maxwell system is lower bounded by a positive constant $C$, then $\mathrm{Ric}=\eta+fg\geq C>0$,
thus $M$ is compact and has finite fundamental group.

Now we try to examine the case of $R-|F|^{2}=0$. In this situation, $\mathrm{Ric}=\eta\geq0$.
Thanks to B\"{o}hm and Wilking's work on nonnegatively curved  manifolds (cf. \cite{BW07}), we obtain
a splitting theorem for the Einstein-Maxwell systems with nonnegative sectional curvature. Although
the Ricci Yang-Mills is a natural geometric flow related to the Einstein-Maxwell systems (cf.\cite{St07},\cite{Yo08}),
we shall use the Ricci flow instead of Ricci Yang-Mills flow since the Einstein-Maxwell system with $f=0$ is
static under the Ricci Yang-Mills flow and we get nothing.

Recall the Ricci flow is a system which evolves the metrics under their Ricci direction, i.e.
$$\frac{\partial}{\partial t}g_{ij}=-2R_{ij}.$$
It is now well known that the nonnegativity of Ricci curvature is non preserved under Ricci flow (cf.\cite{BW07},\cite{Kn06},\cite{Ma11}).
Using Hamilton's maximum principal (see \cite{Ha86},\cite{CC08},\cite{CL04}), B\"{o}hm and Wilking constructed an invariant subset
and obtain that the nonnegativity of Ricci curvature is preserved in a short
 time interval if the initial manifold is compact and with nonnegative sectional curvature.
\begin{lemma}[B\"{o}hm and Wilking \cite{BW07} Proposition 2.1]\label{lem:4.1}
Suppose $(M,g_{0})$ is a compact nonnegative curved n-manifold with the Riemannian curvature
bounded by $C>0$. Then there is a constant $\epsilon$, such that the solution $g(t)$ of the Ricci flow
with initial metric $g_{0}$ exists on $[0,\epsilon]$ and $\mathrm{Ric(g(t))}\geq0$ for all $t\in[0,\epsilon]$.
\end{lemma}

Now we use the Uhlenbeck's trick(cf.\cite{CLN06}). Let $(M,g(t))$ be a solution to the Ricci flow.
Suppose $\iota_{0}:E\rightarrow TM$ is an isomorphism from the vector bundle $E$ to the tangent
bundle $TM$. Define a 1-parameter family of bundle isomorphisms
$\iota (t):E\rightarrow TM$ by
\begin{equation}\label{eq:4.1}
\left\{\begin{array}{c}
\frac{\mathrm{d}}{\mathrm{d}t}\iota(t)=\mathrm{Ric}(t)\circ\iota(t)\\
\iota(0)=\iota_{0}.
\end{array}
\right.
\end{equation}
Let $\{e_{a}$ be a frame of $E$ and $h=\iota(t)^{*}g(t)$. It is easy to compute that $h$ is independent of $t$.
Now we state our splitting theorem as following.
\begin{theorem}\label{thm:4.1}
Let $(M,g)$ be a compact nonnegatively curved $(2n+1)$-manifold. Assume $F$ is a $2$-form on $M$ so that
$$\mathrm{Ric}+F\circ F=0.$$
Then the universal covering $\tilde{M}$ splits off a line, i.e. $\tilde{M}=N\times\mathbb{R}$
with product metric and $N$ is an $2n$-dimensional manifold with a metric of nonnegative sectional curvature.
\end{theorem}
Before we prove this theorem, we need a lemma first.
\begin{lemma}\label{lem:4.2}
Let $(M,g)$ be a $(2n+1)$-manifold and $F$ be a $2$-form on $M$. $\eta=-F\circ F$ is nonnegative
definite and there is a smooth vector field $v$ so that $\eta(v,v)=0$.
\end{lemma}
\begin{proof}
At any point $p\in M$, we can choose local coordinates such that the matrix of $F$ is skew-diagonalized at $p$ as follows,
\begin{equation*}
\left(
\begin{array}{cccccc}
& \mu _{1} &  &  &  &  \\
-\mu _{1} &  &  &  &  &  \\
&  & \ddots  &  &  &  \\
&  &  &  & \mu _{n} &  \\
&  &  & -\mu _{n} &  &  \\
&  &  &  &  & 0%
\end{array}%
\right)
\end{equation*}
Then
\begin{equation*}
\eta =\left(
\begin{array}{cccccc}
\mu _{1}^{2} &  &  &  &  &  \\
& \mu _{1}^{2} &  &  &  &  \\
&  & \ddots  &  &  &  \\
&  &  & \mu _{n}^{2} &  &  \\
&  &  &  & -\mu _{n}^{2} &  \\
&  &  &  &  & 0%
\end{array}%
\right)
\end{equation*}
It is easy to see $\eta$ is nonnegatively definite and has a zero eigenvalue.
Since $\eta$ is smooth, the eigenvector field with respect to zero is smooth.
\end{proof}

Now we are ready to prove the splitting theorem.

\begin{proof}
[Proof of Theorem \ref{thm:4.1}]
Using Uhlenbeck's trick, we consider the Ricci flow start with $(M,g)$. Choose
orthonormal frame $\{e_{a}\}$ for $E$ such that $\iota^{*}(t)\mathrm{Ric}(t)$
are diagonalized. Under (\ref{eq:4.1}),
$$\frac{\partial}{\partial t}R_{aa}=\Delta R_{aa}+2R_{abad}R_{bd}.$$
Choose $H>0$, and consider the modified Ricci tensor
$$\tilde{\mathrm{Ric}}(t)\doteq e^{tH}\mathrm{Ric}(g(t)).$$
\begin{eqnarray*}
\frac{\partial}{\partial t}\tilde{R}_{aa}&=&e^{tH}(HR_{aa}+\Delta R_{aa}+2R_{abab}R_{bb})\\
&\geq&e^{tH}\Delta R_{aa}=\Delta\tilde{R}_{aa},
\end{eqnarray*}
for $t\in [ 0,\delta ]$. Here $\delta$ is a small constant. Denote
$\epsilon_{0}=\mathrm{min}\{\epsilon,\delta\}$, where $\epsilon$ is obtained in Lemma \ref{lem:4.1}.
Let $v$ denote a smooth vector filed on $M$ depending smoothly on $t\in[0,\epsilon_{0}]$ with
$\widetilde{\mathrm{Ric}}(v,v)=0$. Then
\begin{eqnarray*}
0 &=&\left( \frac{\partial }{\partial t}\widetilde{\mathrm{Ric}}\right) (v,v) \\
&\geq &2\sum_{a=1}^{n}\widetilde{\mathrm{Ric}}(\nabla _{a}v,\nabla _{a}v) \\
&\geq &0.
\end{eqnarray*}
This means the kernel of the Ricci curvature is invariant under paralle translation.
By Lemma \ref{lem:4.2} there is a vector field such that $\mathrm{Ric}(v,v)=\eta(v,v)=0$.
The universal covering $\tilde{M}$ splits off a line.
\end{proof}
\section{Rigidity theorems for Einstein-Maxwell systems}
In this section, we shall use the Bochner formula to obtain
some rigidity theorems for Einstein-Maxwell systems.

\begin{lemma}(cf.\cite{LW} chapter 8)\label{lem:5.1}
Assume $(M,g)$ is a $n$-dimensional compact Riemannian manifold.
Let $T=\mathrm{Rm}-\lambda \mathrm{I}\in \Gamma(\bigwedge^{2}T^{*}M\bigotimes E),$
where $\lambda \in \mathbb{R}$ is a constant. If $\mathrm{d}^{*}\mathrm{Rm}=0,$
then for any $q>n/2,$ there exist a constant $C(n,q,\lambda, C_{S})>0$,
so that
$$|T|\leq C(n,q,\lambda, C_{S})\|T\|_{q}$$
And there is $0<\epsilon(n,\lambda, C_{S})<1$,
such that if $\|T\|_{n/2}\leq \epsilon,$ then
$$|T|\leq C(n,\lambda, C_{S})\|T\|_{n/2}.$$
Where $C_{S}$ is the Sobolev constant.
\end{lemma}

The Bochner technique implies that an Einstein-Maxwell system with positive curvature
operator must be an Einstein manifold with $F=0$ since $F$ is a harmonic
form (cf.\cite{Li09}).

From the observation above and lemma \ref{lem:5.1}, we have the following
rigidity theorem which is a generalization of Einstein manifolds.
\begin{theorem}\label{thm:5.1}
Given $\lambda>0, \delta>0$, there is a constant $\epsilon(n,\lambda,\delta)>0$ such that if $(M,g,F)$ is an Einstein-Maxwell system with
\begin{itemize}
\item[(\romannumeral1)]$R-|F|^{2}\geq\delta$ and $\nabla F=0$,
\item[(\romannumeral2)]$\|\mathrm{Rm}-\lambda \mathrm{I}\|_{\frac{n}{2}}\leq\epsilon,$
\end{itemize}
then $g$ has constant sectional curvature and $F=0$.
\end{theorem}
\begin{proof}
The positive lower bound for the Ricci tensor gives an upper bound
for the diameter by virtue of Myers' theorem.
Thanks to Gromov and Gallot \cite{Ga88} we know that upper diameter bounds and
lower Ricci curvature bounds give bounds for $C_{s}$.
Then by lemma \ref{lem:5.1}, the curvature operator has eigenvalues close
to $\lambda>0$ and hence are all positive for small $\epsilon$.

Thus by the observation above and Tachibana's theorem (\cite{Ta74}), which states that a compact
oriented Riemannian manifold with $\mathrm{d}^{*}\mathrm{Rm}=0$
and $\mathrm{Rm}>0$ must have constant sectional curvature. $\nabla F=0$ implies $\mathrm{d}^{*}\mathrm{Rm}=0$.
Then the theorem follows easily.
\end{proof}

S.Goldberg and S.Kobayashi in \cite{GK67} proved that a compact K\"{a}hler-Einstein
manifold with positive orthogonal bisectional curvature has constant
holomorphic sectional curvature.
We can extend this theorem to K\"{a}hler Einstein-Maxwell systems with positive quadratic orthogonal
bisectional curvature.
A K\"{a}hler Einstein-Maxwell system is an Einstein-Maxwell system $(M,g,F)$ with
the metric $g$ a K\"{a}hler metric and $F$ a harmonic $(1,1)$-form.

In \cite{GZ10}, Gu and Zhang proved that if a K\"{a}hler manifold has nonnegative
orthogonal bisectional curvature, then all harmonic $(1,1)$-forms are parallel.
Using the standard Bochner technique, A.Chau and L.Tam\cite{CT11} generalized this result to K\"{a}hler
manifolds with nonnegative quadratic orthogonal bisectional curvature.
\begin{lemma}[cf.\cite{CT11},\cite{GZ10}]\label{lem:5.2}
If $(M,g)$ has nonnegative quadratic orthogonal holomorphic bisectional curvature,
then all harmonic $(1,1)$-forms are parallel.
\end{lemma}
\begin{lemma}[see \cite{GH76}for example]\label{lem:5.3}
If $\xi$ is a $(p,q)$-form on a closed K\"{a}hler manifold, and $\xi$ is $\mathrm{d}$-,$\partial$- or
$\overline\partial$-exact, then there is a $(p-1,q-1)$-form $\varrho$
$$\xi=\partial\overline\partial\varrho.$$ If $p=q$, and $\xi$ is real, then we may take $\sqrt{-1}\varrho$
is real.
\end{lemma}
Based on these results, we have the following theorem.
\begin{theorem}\label{thm:5.2}
Assume $(M,g)$ is a closed K\"{a}hler manifold of complex dimension $n\geq 2$, $F$ is
a real $(1,1)$ form on $M$. If $(g,F)$ satisfies the Einstein-Maxwell equations and the metric $g$
has nonnegative quadratic orthogonal holomorphic bisectional curvature.
Then $F$ and the Ricci curvature $\mathrm{Ric}$ are parallel.
Furthermore,if $b_{1,1}(M)= dim H^{1,1}(M; \mathbb{R})=1$, then
$g$ has constant holomorphic sectional curvature.
\end{theorem}
\begin{proof}
By the theorem of A.Chau and L.Tam, it is easy to know that $F$ is parallel since $F$
is harmonic. Thus we have that $|F|$ is constant and then the scalar curvature $R$ is a positive constant
thanks to lemma \ref{lem:2.1}.
Since the nonnegativity of quadratic orthogonal holomorphic bisectional curvature implies $R\geq0$.
We also know that $\eta$ is parallel, so do $\overset{\circ}{\eta}$ and $\overset{\circ}{\mathrm{Ric}}$.
Thus the Ricci curvature is parallel consequently.
If further $b_{1,1}(M)= 1$, then there is a constant $\kappa$ so that the Ricci form
$$[\rho]=\kappa[\omega].$$
Then lemma \ref{lem:5.3} \ tells us that there exists a real function $h$ on $M$, such that
$$\rho=\kappa\omega+\sqrt{-1}\partial\overline{\partial}h.$$
Take trace of both sides, we get
$$\frac{1}{2}\Delta h=\Delta_{\overline{\partial}}h=R-n\kappa.$$
The maximum principal gives us that $h$ is constant and $g$ is a K\"{a}hler-Einstein metric
$$\mathrm{Ric}=\frac{1}{n}Rg.$$
Thus the solution $(g,F)$ to Einstein Maxwell equations reduce to an Einstein metric and
a harmonic form. By using the theorem of S.Goldberg and S.Kobayashi,
we know that $(M,g)$ has constant holomorphic sectional curvature.
\end{proof}
S.Brendle \cite{Br08} showed that if $(M,g)$ has nonnegative isotropic curvature and $Hol^{0}(M,g)=U(n)$,
then $(M,g)$ has positive orthogonal bisectional curvature.(also see \cite{Se07} .)
By Theorem 2.1 and Corollary 2.2 in \cite{GZ10}, a K\"{a}hler manifold with
positive orthogonal holomorphic bisectional curvature must have
$b_{1,1}(M)=\dim H^{1,1}(M)=1$ and $C_{1}(M)>0$. In particular, all real harmonic $(1,1)$-forms are parallel.
Combining these with theorem 1.9  in \cite{Ch06}, we have the following corollary.
\begin{corollary}
Assume $(M,g)$ is a smooth oriented closed manifold with real dimension $2n\geq 4$, and $F$
is a $2$-form such that $(g,F)$ is a solution to the Einstein-Maxwell equations.
If $(M,g)$ has nonnegative isotropic curvature and $Hol^{0}(M,g)=U(n)$, then
$M$ is biholomorphic to $\mathbb{C}P^{n}$ and $g$ has constant holomorphic sectional curvature.
\end{corollary}

We mention here that the condition $Hol^{0}(M,g)=U(n)$ means $g$ is a generic K\"{a}hler metric.
A rigidity theorem for generic K\"{a}hler surface with constant scalar curvature follows.
\begin{corollary}
Let $(M,g)$ be a smooth oriented compact 4-manifold with constant scalar curvature. If $(M,g)$ has
nonnegative isotropic curvature and $Hol^{0}(M,g)=U(2)$, then $M$ is biholomorphic
to $\mathbb{C}P^{2}$ and $g$ has constant holomorphic sectional curvature.
\end{corollary}

\bigskip \textbf{Acknowledgement. }I am grateful to Professor F.Q. Fang and Professor L. Ni for
their brilliant guidance and stimulations. Thanks also go to A. Zhu and K. Wang for a lot of
conversations.\bigskip

\bibliographystyle{amsplain}

\end{document}